%% file: IHCC.tex
\documentclass[12pt]{amsart}

\usepackage{microtype}

\usepackage{mathpazo} %different font

\usepackage{hyperref}
\hypersetup{
    colorlinks=true,
    linkcolor=Mahogany,      
    urlcolor=cyan,
    citecolor=blue,
}
\usepackage[hyperpageref]{backref}  %Adds References to the page an article is cited on to the bibliography

\usepackage[abbrev,alphabetic]{amsrefs} %Different Bibliography layout

\usepackage{amssymb} 
\usepackage[dvipsnames]{xcolor}
\usepackage{multirow}
\usepackage{pgf}
\usepackage{tikz-cd}
\usepackage[style=english]{csquotes}
\usepackage{url}

\newtheorem{thm}{Theorem}[section] 
\newtheorem*{thm*}{Theorem}

\newtheorem{lem}[thm]{Lemma}

\theoremstyle{definition} 

\newtheorem{ex}[thm]{Example}

\theoremstyle{remark}

\numberwithin{equation}{section}

\newcommand{\codim}[0]{\operatorname{codim}}

\newcommand{\reg}{{\rm{reg}}}
\newcommand{\sing}{{\rm{sing}}}

\newcommand{\Aut}{\mathrm{Aut}}

\newcommand{\Q}{\mathbb{Q}}
\newcommand{\mP}{\mathbb{P}}
\newcommand{\B}{\mathbb{B}}

\makeatletter
\newcommand*\bigcdot{\mathpalette\bigcdot@{.5}}
\newcommand*\bigcdot@[2]{\mathbin{\vcenter{\hbox{\scalebox{#2}{$\m@th#1\bullet$}}}}}
\makeatother

\title{Characterising Ball Quotients through their (higher) Chern Numbers}

\author{Niklas M\"uller}
\address{Department of Mathematics, Universit\"at Freiburg,
Ernst-Zermelo-Str. 1, 79104 Freiburg, Germany.}
\email{{\tt niklas.mueller@math.uni-freiburg.de}}

\date{\today}

\subjclass[2020]{Primary 32Q30, Secondary 14E30, 32Q45.}

\keywords{Uniformisation, Chern classes, minimal varieties, Hermitian symmetric domains}

\begin{document}

%%% Titlesequence

    \begin{abstract}
        In this short note we provide a characterisation of ball quotients among all minimal smooth projective varieties of general type purely in terms of their characteristic numbers. This generalises earlier work of Miyaoka, Yau and Greb--Kebekus--Peternell--Taji.
    \end{abstract}
    
    \maketitle

%%% TOC    

    \setcounter{tocdepth}{1}
    \tableofcontents

%%% Main content
    
    \input{1-Introduction}

\input{2-Preliminaries}

\input{3-StringyEulerNumbers}

%%% Bibliography

\bibliography{IHCC.bib}

 \end{document}

%% file: 1-Introduction.tex
\renewcommand{\thethm}{\thesection.\arabic{thm}}
    \renewcommand{\theequation}{\thesection.\arabic{equation}}

    \section{Introduction}
\label{section-Introduction}

A \emph{ball quotient variety} is a compact complex manifold $X$ whose universal cover is biholomorphic to the complex unit ball $\mathbb{B}^n= \{z\in \mathbb{C}^n \ | \ \|z\|^2 < 1\}$. In other words, $X\cong \mathbb{B}^n/\Lambda$, where $\Lambda \subseteq \Aut(\mathbb{B}^n) = \mathrm{PU}(n,1)$ is a torsion-free and cocompact lattice. In this case, it is well-known that the canonical line bundle $\omega_X$ is positive; in particular, ball quotient varieties are always projective. 

A famous source of examples of ball quotients are fake projective planes, i.e.\ smooth projective surfaces $S$ which are not isomorphic to $\mP^2$ but satisfy $H^i(S, \Q) \cong H^i(\mP^2, \Q)$; see \cites{PrasadYeung_FakeProjectivePlanes, CartwrightSteger_EnumerationOfAllFakeProjectivePlanes} for a complete classification of the corresponding subgroups $\Lambda \subseteq \mathrm{PU}(2,1)$. Following work of Deligne--Mostow \cite{DeligneMostow_BallQuotients}, there has been a surge of interest in finding explicit complex geometric constructions of ball quotient varieties, see \cite{Hunt_GeographyOfComplexSurfacesAnd3Folds} for an overview. Recently, a new example has been discovered by Deroin--March{\'e} \cite[Section 6]{DM_ToledoInvariantsTQFT}. In this context, to show that a given variety is a ball quotient, the following criterion of Miyaoka proved very useful:

\begin{thm*}\emph{(Miyaoka \cite{Miyaoka_ChernNumbersSurfacesGeneralType})}
    Let $X$ be a smooth complex projective surface such that $K_X$ is big and nef. Then $X$ is isomorphic to a quotient of the complex unit ball $\mathbb{B}^2$ if and only if 
    \[
    3c_2(X) - c_1(X)^2= 0.
    \]
\end{thm*}
More generally, it is known by work of Chen--Ogiue \cite{C075}, Aubin \cite{Aub78} and Yau \cite{Yau_KEMetrics} that a smooth projective variety of dimension $n$ with ample canonical bundle $K_X$ is a ball quotient if and only if 
\begin{equation}
    \Big( 2(n+1)c_2(X) - n c_1(X)^2\Big)\cdot K_X^{n-2} = 0.
    \label{eq-GKPT}
\end{equation}
Under the assumption that $K_X$ is merely big and nef, Greb--Kebekus--Peternell--Taji \cite{GKPT_MY_Inequality_Uniformisation_of_Canonical_models} showed that if \eqref{eq-GKPT} holds, then the \emph{canonical model} $X_{\mathrm{can}}$ is a, possibly singular, ball quotient variety. Note, however, that in this case $X$ itself may or may not be isomorphic to a ball quotient \cite[Thm.\ D]{Mul_Thesis}. The primary goal of this note is to extend \cite{GKPT_MY_Inequality_Uniformisation_of_Canonical_models} to give a complete characterization of ball quotients among all minimal smooth projective varieties of general type, purely in terms of their Chern numbers: 
\begin{thm}
    \label{thm-main-result}
    Let $X$ be a smooth complex projective variety of dimension $n$ such that $K_X$ is big and nef. Then $X$ is isomorphic to a quotient of the complex unit ball $\mathbb{B}^n$ if and only if 
    \[
    \left(  (n+1)^i\cdot  c_i(X) -  \binom{n+1}{i} c_1(X)^i \right) \cdot K_X^{n-i} = 0, \qquad \forall i=1,\ldots, n.
    \]
\end{thm}
The author is grateful towards Ben McKay for pointing out that Theorem \ref{thm-main-result} is intimately related to the existence of holomorphic normal projective connections. Indeed, if $X$ is a complex manifold which admits such a connection, then the series of equalities in Theorem \ref{thm-main-result} holds, see \cite[Thm.\ 3.1]{KO80}. Under this additional assumption, the conclusion that $X\cong \B^n$ was obtained previously by Jahnke--Radloff \cite{JR02}. On the other hand, the assumptions in Theorem \ref{thm-main-result} are a priori much weaker than the existence of such a connection.

We will deduce Theorem \ref{thm-main-result} from the following more general result:
\begin{thm}
\label{thm-inequalities-higher-Chern-classes}
Let $X$ be a smooth complex projective variety of dimension $n$ such that $K_X$ is big and nef. Assume, moreover, that 
    \begin{equation}
        c_i(X)\cdot K_X^{n-i} = \frac{1}{(n+1)^i}\cdot \binom{n+1}{i} \cdot c_1(X)^i\cdot K_X^{n-i}, \qquad \forall i=1, \ldots, k-1,
           \label{eq-uniformisation-canonical-models-1}
    \end{equation}
    for some integer $2\leq k \leq n$. Then
    \begin{equation}
        c_k(X)\cdot K_X^{n-k} \geq \frac{1}{(n+1)^k}\cdot \binom{n+1}{k} \cdot  c_1(X)^k\cdot K_X^{n-k}.
           \label{eq-uniformisation-canonical-models-2}
    \end{equation}
    Moreover, the equality holds in \eqref{eq-uniformisation-canonical-models-2} if and only if $X_{\mathrm{can}}\cong \B^n/\Lambda$ is isomorphic to a (possibly singular) ball quotient variety and 
    the natural birational map
    \[
       f\colon X\rightarrow X_{\mathrm{can}}
    \]
    is an isomorphism over an open subset $U\subseteq X_{\mathrm{can}}$ with $\codim (X_{\mathrm{can}}\setminus U) \geq k+1$.
\end{thm}
Note that \eqref{eq-uniformisation-canonical-models-1} is always satisfied when $k=2$; in this case, Theorem \ref{thm-inequalities-higher-Chern-classes} coincides precisely with \cite[Thm.\ 1.1, Thm.\ 1.2]{GKPT_MY_Inequality_Uniformisation_of_Canonical_models}. 

In recent years, inequalities between the first and second Chern class of  projective varieties have been thoroughly investigated, see for example \cite{GKT_OverviewUniformisation} or \cite[Introduction]{Mul_Thesis} for a comprehensive survey. However, little is known so far concerning the higher Chern classes. The boundedness of Chern classes of algebraic varieties of general type has been investigated in increasing generality by Kollár--Matsusaka \cite{KollarMatsusaka_RRInequalities}, Hunt \cite{Hunt_GeographyOfComplexSurfacesAnd3Folds}, Lu \cite{LU_RRTypeInequalities}, Du--Sun \cite{DuSun_ChernClassInequalities}, Ping--Zheng \cite{PF_ChernCLassesPolarisedManifolds} and Lu--Xiao \cite{LuXiao_InequalitiesOfRiemannRochType}; in particular, it is known \cite{DuSun_ChernClassInequalities} that there exists a constant $C_n>0$ such that
\[
\left| \frac{c_\lambda(X)}{c_1(X)^n}\right| \leq C_n
\]
for any smooth projective variety $X$ of dimension $n$ such that $K_X$ (or $-K_X$) is ample and any monomial Chern class of top degree $c_\lambda(X)$. However, the works above do not concern themselves with the question of finding the optimal value for $C_n$. Note that examples of Schreieder--Tasin \cite{SchreiederTasin_ChernClassesUnbounded} show that the boundedness fails if one does not require positivity of the canonical bundle, even when fixing the underlying diffeomorphism type of $X$. 

On the other hand, for threefolds $X$ fibred in minimal smooth projective surfaces, Lu--Tan--Zuo \cite{LTZ_CanonicalClassInequalityFibredThreefolds} showed that $18c_3(X) > c_1(X)^3$, while for smooth projective threefolds with ample canonical bundle $\omega_X$, Sun \cite{Sun_bogomolovtypeinequalitiesfrobenius} recently showed that
\[
ch_3(\Omega_X \otimes \omega_X^{\otimes 6}) \leq \frac{19^3}{54} \cdot K_X^3.
\]
The author is unaware of any further research on the behaviour of the higher Chern classes of minimal varieties of general type.
It seems interesting to further investigate this subject.

Throughout this text, we work over the field $\mathbb{C}$ of complex numbers. We employ the standard notions and conventions from \cite{kollarMori_BirationalGeometry, lazarsfeld_PositivityI}.

\subsection*{Acknowledgments}

I am grateful towards Daniel Greb for many interesting discussions on Chern classes and for his support. Moreover, I would like to thank Marc Levine for suggesting me to learn about stringy invariants.
The hospitality of Stefan Kebekus and the stimulative working environment at the Universit\"at Freiburg, where parts of this article were written, is much appreciated. 
Finally, I would like to thank the referee for their careful reading of this text and for their extremely helpful suggestions!

During the creation of this article, I was partially supported by the DFG research training group 2553 ‘Symmetries and classifying spaces:
analytic, arithmetic and derived’.

%% file: 2-Preliminaries.tex
\section{Stringy Euler number}
    \label{sec-preliminaries}

    In this section, we recall the definition of the so-called \emph{stringy Euler number}, which will be useful in proving Theorem \ref{thm-inequalities-higher-Chern-classes}. The concept of stringy invariants first arose in physics \cite{DHVW_StringsOnOrbifoldsI, DHVW_StringsOnOrbifoldsII, ZaslowStringyInvariants} and was later adapted by Batyrev \cite{Batyrev_MotivicIntegrationStringyEulerNumbers} to the context of birational geometry: Let $X$ be a smooth algebraic variety and let $D$ be a $\Q$-divisor on $X$ with simple normal crossing support (we call $(X, D)$ a \emph{log smooth pair}). Assume that $(X, D)$ is sub-klt \cite[Def.\ 2.34]{kollarMori_BirationalGeometry}; in other words, writing $D = \sum_{i=1}^s d_i \cdot D_i$, we assume that $d_i<1$.
    For any subset $I\subseteq \{1, \ldots, s\}$ denote
        \[
        D_I = \left\{ 
        \begin{matrix}
            \bigcap_{i\in I} D_i,& I \neq \emptyset \\
            X, & I = \emptyset
        \end{matrix}
        \right.
        \quad \mathrm{and}\quad
        D^\circ_I = D_I\setminus \bigcup_{i \in \{1, \ldots, s\}\setminus I} D_I \cap D_i.
    \]
    The \emph{stringy Euler number} of $(X, D)$ is the rational number
    \[
    e_{\mathrm{Str}}(X, D) := \sum_{I\subseteq \{1, \ldots, s\}} \left(\prod_{i\in I} \frac{1}{1-d_i}\right) \ e(D_I^\circ),
    \]
    where $e(D_I^\circ)$ denotes the usual topological Euler characteristic of $D_I^\circ$. 
    \begin{ex}
        Let $X$ be a smooth variety. Then $e_{\mathrm{Str}}(X) := e_{\mathrm{Str}}(X, 0) = e(X)$, cf.\  \cite[p.\ 142]{Fulton_ToricVarieties}.
    \end{ex}
    The definition of stringy Euler number can be extended to arbitrary sub-klt pairs using the following fundamental Theorem of Batyrev:

    \begin{thm}\emph{(Batyrev \cite[Thm.\ 1.5]{Batyrev_MotivicIntegrationStringyEulerNumbers})}
    \label{thm-stringy-euler-number-welldefined}
        Let $\phi\colon (X_1, D_1) \rightarrow (X_2, D_2)$ be a proper birational morphism between log smooth sub-klt pairs such that $K_{X_1} + D_1 \sim_{\mathbb{Q}} \phi^*(K_{X_2}+D_2)$.
        Then
        \[
        e_{\mathrm{Str}}(X_1, D_1) = e_{\mathrm{Str}}(X_2, D_2).
        \]
    \end{thm}
    Indeed, given a sub-klt pair $(X, D)$, choose a proper, birational morphism $\phi\colon (\widetilde{X}, \widetilde{D}) \rightarrow (X, D)$ such that $(\widetilde{X}, \widetilde{D})$ is log smooth and such that $K_{\widetilde{X}}+ \widetilde{D} \sim_{\mathbb{Q}} \phi^*(K_{X} + D)$. Then the rational number
    \[
        e_{\mathrm{Str}}(X, D) := e_{\mathrm{Str}}(\widetilde{X}, \widetilde{D})
    \]
   is called the \emph{stringy Euler number} of $(X, D)$; by Theorem \ref{thm-stringy-euler-number-welldefined}, it is well-defined and independent of the choice of $\phi$.
    
    Strictly speaking, Batyrev only proves Theorem \ref{thm-stringy-euler-number-welldefined} in case the exceptional locus of $\phi$ is a divisor with simple normal crossings. The general case is readily reduced to this one by passing to a higher resolution. Alternatively, the statement of Theorem \ref{thm-stringy-euler-number-welldefined} follows from \cite[Thm.\ 0.1]{FLNU_StringyChernClasses}.

%% file: 3-StringyEulerNumbers.tex
\section{Proof of main results}

\label{sec-Inequalities-higher-chern-classes}

In this section, we prove Theorems \ref{thm-main-result} and  \ref{thm-inequalities-higher-Chern-classes}. The following Lemma, which is proved using the concept of stringy Euler numbers, is the key ingredient:
\begin{lem}
    \label{lem-euler-characteristic-crepant-resolution}
        Let $S$ be a normal projective variety of dimension $n$ with isolated quotient singularities only. Assume that there exists a finite quasi-\'etale Galois cover $\pi\colon S'\rightarrow S$ by a smooth variety $S'$. Assume, moreover, that there exists a proper birational morphism $f\colon Z \rightarrow S$ such that $Z$ is a smooth variety and such that $K_Z \sim_{\mathbb{Q}} f^*K_S$. Then
        \[
        c_n\left(Z\right) \geq c_n(S) := \frac{c_n\big(S'\big)}{\deg \pi},
        \]
        with equality if and only if $S$ is smooth. In this case, $f\colon Z \rightarrow S$ is an isomorphism.
\end{lem}    
    \begin{proof}
        Denote by $G$ the Galois group of $\pi\colon S'\rightarrow S$ and set $d := \deg\pi$.
        By the \emph{Poincar\'e--Hopf theorem} \cite[p.\ 35]{Milnor_TopologyDifferentialTopology},
        \begin{equation}
          c_n\left(Z\right) = e\left(Z\right),
          \quad \mathrm{and} \quad
        c_n\left(S'\right) = e\left(S'\right). 
        \label{eq-euler-crepant-1}
        \end{equation}
        Since $K_Z \sim_{\mathbb{Q}} f^*K_S$, it follows from \cite[Cor.\ 2.31.(2), Cor.\ 2.63]{kollarMori_BirationalGeometry} that $f$ is an isomorphism away from $S^{\sing} = \{x_1, \ldots, x_s\}$.
        For each $1\leq i \leq s$, let $E_i := (f^{-1}(x_i))_{\mathrm{red}}$. Choose points $x_i'\in S'$ with $\pi(x_i') = x_i$, denote by $G_i\subseteq G$ the stabiliser group of $x_i'$ and let $V_i := T_{x_i'}S'$. By Luna's \'etale slice theorem \cite{Luna73}, the germ $(S, x_i)$ is \'etale locally isomorphic to $(V_i/G_i, 0)$. 
        For each $1\leq i \leq s$, let $f_i\colon Z_i \rightarrow V_i/G_i$ be a proper birational morphism such that $Z_i$ is smooth, $K_{Z_i} \sim_{\mathbb{Q}} f_i^*K_{V_i/G_i}$ and $(f^{-1}(0))_{\mathrm{red}} \cong E_i$. For example, if, locally near $x_i$, $f$ is given by the blow up of the ideal $\mathcal{I}_i\subseteq \mathcal{O}_S$ \cite[Thm. II.7.17]{Har77}, then we make take $f_i$ to be the blow up of the corresponding ideal in $\mathcal{O}_{V_i/G_i}$. In any case, by definition, 
        \[
        e_{\mathrm{Str}}(S) 
        := e\left(Z\right) 
        = e(S_{\mathrm{reg}}) + \sum_{i=1}^s e(E_i),
        \] 
        where we used \cite[p.\ 142]{Fulton_ToricVarieties}.
        Note that $Z_i$ deformation retracts onto $E_i$. Consequently,
        \[
        e(E_i)
        = e(Z_i)
        =: e_{\mathrm{Str}}(V_i/G_i).
        \]
        By \cite[Thm.\ 1.9]{Batyrev_MotivicIntegrationStringyEulerNumbers},
        \begin{equation*}
            e_{\mathrm{Str}}(V_i/G_i) = \# \Big\{ \big\{ hgh^{-1} \big| h\in G_i \big\} \Big| g\in G_i \Big\} =: m_i
        \end{equation*}
        is the number of conjugacy classes of $G_i$ under the action on itself by conjugation. In summary,
        \begin{equation}
           e(Z) = e_{\mathrm{Str}}(S) = e(S_{\mathrm{reg}}) + \sum_{i=1}^s m_i.
        \label{eq-euler-crepant-1.5}
        \end{equation}
        On the other hand, as $\pi$ is \'etale of degree $d$ over $S_{\reg}$,
        \begin{equation}
            e\big(S_{\reg}\big) 
        = \frac{1}{d}\ e\Big(\pi^{-1}\big(S_{\reg}\big)\Big) 
        = \frac{1}{d}\ \left(e(S') -  \sum_{i=1}^s \# \left(\pi^{-1}(x_i)\right)\right)
        = \frac{e(S')}{d} -  \sum_{i=1}^s \frac{1}{\#G_i}.
        \label{eq-euler-crepant-4}
        \end{equation}
        Combining \eqref{eq-euler-crepant-1}, \eqref{eq-euler-crepant-1.5} and \eqref{eq-euler-crepant-4}, we obtain the formula
        \begin{align*}
                c_n\left(Z\right) 
          = e\left(Z\right) 
        = e_{\mathrm{Str}}(S) 
        = e(S_{\reg}) + \sum_{i=1}^s m_i
        = \frac{c_n\left(S'\right)}{d} + \sum_{i=1}^s \left(m_i - \frac{1}{\# G_i}\right).
        \end{align*}
        Finally, observe that $m_i \geq 1$ and $\# G_i\geq 1$, with equality if and only if $G_i = \{1\}$. Hence,
        \[
        m_i - \frac{1}{\# G_i} \geq 0,
        \]
        with equality if and only if $G_i = \{1\}$. This, in turn, is equivalent to $\pi$ being \'etale over $x_i$, that is $x_i\in S$ being a smooth point. In summary,
        \[
        c_n\left(Z\right) \geq \frac{c_n\left(S'\right)}{d}
        \]
        with equality if and only if $S$ is smooth. In the latter case, $f$ is necessarily an isomorphism, see \cite[Cor.\ 2.31.(2), Cor.\ 2.63]{kollarMori_BirationalGeometry}. This concludes the proof.
    \end{proof}

    \begin{proof}[Proof of Theorem \ref{thm-inequalities-higher-Chern-classes}]
        We proceed by induction on $k\geq 2$. The case $k=2$ is just \cite{GKPT_HarmonicMetricsUniformisation}. Now, assume that $k>2$. According to \cite{GKPT_HarmonicMetricsUniformisation}, $B := X_{\mathrm{can}}\cong \B^n/\Lambda$ is a, possibly singular, ball quotient variety. By the induction hypothesis, $f\colon X\rightarrow B$ is an isomorphism over an open subset with complement of codimension at least $k$, and we need to show that
        \begin{equation}
            c_k(X)\cdot K_X^{n-k} \geq \frac{1}{(n+1)^k} \binom{n+1}{k}\cdot c_1(X)^k\cdot K_X^{n-k},
            \label{eq-uniformisation-canonical-models-to-prove}
        \end{equation}
        with equality if and only if $f\colon X\rightarrow B$ is an isomorphism over an open subset with complement of codimension at least $k+1$.

        Indeed, choose a finite quasi-\'etale Galois cover $\pi\colon B' \rightarrow B$ by a smooth ball quotient variety $B'$, of degree $d$ say.
        Fix a sufficiently large integer $\ell \geq 1$ and general hypersurfaces $A_1, \ldots, A_{n-k} \in |\ell K_{B}|$. Set
        \[
        S := A_1\cap \ldots \cap A_{n-k}, \quad 
        Z := f^{-1}(S), \quad \mathrm{and} \quad
        S' := \pi^{-1}(S).
        \]
        By construction, $Z$ is smooth and $f\colon Z \rightarrow S$ is a birational morphism which satisfies $K_Z \sim_{\mathbb{Q}}f^*K_S$ and is an isomorphism away from finitely many points $x_1, \ldots, x_s \in S$. In particular, $S$ has at most finitely many isolated singularities. Moreover, $\pi\colon S'\rightarrow S$ is a finite quasi-\'etale Galois cover, of degree $d$, by the smooth variety $S'$. 

        Now, consider the short exact sequence
        \[
        0 \rightarrow \mathcal{T}_{Z} \rightarrow \mathcal{T}_X|_{Z} \rightarrow \mathcal{N}_{Z/X} \rightarrow 0.
        \]
        Set $D_i := f^{-1}(A_i)$. Since the total Chern class is multiplicative on short exact sequences we see that
        \begin{align*}
            \begin{split}
                & \ell^{n-k} \cdot c_k(X)\cdot K_X^{n-k}
                 = c_k(X)\cdot D_1\cdots D_{n-k} 
                = c_k(X)|_{Z} \\
                & \qquad = \sum_{i=0}^k c_{k-i}(Z) \cdot c_{i}\Big(\mathcal{N}_{Z/X}\Big)
                = \sum_{i=0}^k \ell^i \cdot \binom{k}{i} \ c_{k-i}(Z) \cdot (K_X|_Z)^i.
            \end{split}
        \end{align*}
        Here, we used that $\mathcal{N}_{Z/X} \cong \mathcal{O}(D_1) \oplus \ldots \oplus \mathcal{O}(D_m)$ to obtain the final equality. In any case, a similar argument shows that
        \begin{gather*}
            c_k(B')\cdot K_{B'}^{n-k}
            = \frac{1}{\ell^{n-k}}
            \sum_{i=0}^k \ell^i \cdot \binom{k}{i}\cdot 
            c_{k-i}\big(S'\big)\cdot (K_{B'}|_{S'})^{i}.
        \end{gather*}
        As $f\colon Z\rightarrow S$ is an isomorphism away from finitely many points of $S$, we see that $d\cdot c_{k-i}(Z)\cdot (K_X|_Z)^{i} = c_{k-i}(S')\cdot (K_{B'}|_{S'})^{i}$ for all $i\geq 1$ and, consequently,
        \begin{gather}
            c_k(X)\cdot K_X^{n-k} - \frac{c_k(B')\cdot K_{B'}^{n-k}}{d} 
            = \frac{1}{\ell^{n-k}}\left(c_k(Z)  - \frac{c_k(S')}{d}\right) \geq 0.  
            \label{eq-uniformisation-canonical-models-formula-1}
        \end{gather}
        Here, the inequality on the right is due to Lemma \ref{lem-euler-characteristic-crepant-resolution}.
        On the other hand, since $B'$ is a smooth ball quotient variety, 
        \begin{align}
            c_{k}\big(B'\big)\cdot K_{B'}^{n-k} 
        = \frac{\binom{n+1}{k}}{(n+1)^k} \cdot c_1\big(B'\big)^k \cdot K_{B'}^{n-k}
        = \frac{d \cdot \binom{n+1}{k}}{(n+1)^k}\cdot c_1(X)^k \cdot K_X^{n-k}.
        \label{eq-uniformisation-canonical-models-formula-1'}
        \end{align}
        Putting \eqref{eq-uniformisation-canonical-models-formula-1} and \eqref{eq-uniformisation-canonical-models-formula-1'} together, we deduce that
        \[
        c_k(X)\cdot K_X^{n-k} \geq \frac{1}{(n+1)^k}\binom{n+1}{k}\cdot  c_1(X)^k\cdot K_X^{n-k}.
        \]
        In case of equality, Lemma \ref{lem-euler-characteristic-crepant-resolution} shows that $f\colon Z \rightarrow S$ is an isomorphism. Equivalently, $f\colon X \rightarrow B$ is an isomorphism away from a subset of codimension $k+1$ in $B$, as required.
    \end{proof}

    \begin{proof}[Proof of Theorem \ref{thm-main-result}]
        Follows immediately from Theorem \ref{thm-inequalities-higher-Chern-classes}.
    \end{proof}

%% file: IHCC.bib
@article {Aub78,
    AUTHOR = {Aubin, Thierry},
     TITLE = {{\'E}quations du type {M}onge-{A}mp\`ere sur les vari\'et\'es
              k\"ahl\'eriennes compactes},
   JOURNAL = {Bull. Sci. Math. (2)},
  FJOURNAL = {Bulletin des Sciences Math\'ematiques. 2e S\'erie},
    VOLUME = {102},
      YEAR = {1978},
    NUMBER = {1},
     PAGES = {63--95},
}

@article{Batyrev_MotivicIntegrationStringyEulerNumbers,
AUTHOR = {Batyrev, Victor},
     TITLE = {Non-{A}rchimedean integrals and stringy {E}uler numbers of log-terminal pairs},
   JOURNAL = {J. Eur. Math. Soc. (JEMS)},
  FJOURNAL = {Journal of the European Mathematical Society},
    VOLUME = {1},
      YEAR = {1999},
    NUMBER = {1},
     PAGES = {5--33},
       DOI = {10.1007/s00222-022-01096-y}
}

@Article{CartwrightSteger_EnumerationOfAllFakeProjectivePlanes,
 Author = {Cartwright, Donald and Steger, Tim},
 Title = {Enumeration of the 50 fake projective planes},
 FJournal = {Comptes Rendus. Math{\'e}matique. Acad{\'e}mie des Sciences, Paris},
 Journal = {C. R., Math., Acad. Sci. Paris},
 ISSN = {1631-073X},
 Volume = {348},
 Number = {1-2},
 Pages = {11--13},
 Year = {2010},
}

@article {C075,
    AUTHOR = {Chen, {B.} and Ogiue, Koichi},
     TITLE = {On compact {E}instein-{K}aehler manifolds},
   JOURNAL = {Proc. Amer. Math. Soc.},
  FJOURNAL = {Proceedings of the American Mathematical Society},
    VOLUME = {53},
      YEAR = {1975},
    NUMBER = {1},
     PAGES = {176--178},
}

@article{DHVW_StringsOnOrbifoldsI,
  AUTHOR = {Dixon, L. and Harvey, J. and Vafa,
              C. and Witten, E.},
     TITLE = {Strings on orbifolds {I}},
   JOURNAL = {Nucl. Phys.},
    VOLUME = {B 261},
      YEAR = {1985},
     PAGES = {678--686}
}

@article{DHVW_StringsOnOrbifoldsII,
  AUTHOR = {Dixon, L. and Harvey, J. and Vafa,
              C. and Witten, E.},
     TITLE = {Strings on orbifolds {II}},
   JOURNAL = {Nucl. Phys.},
    VOLUME = {B 274},
      YEAR = {1986},
     PAGES = {285--314}
}

@article {DeligneMostow_BallQuotients,
    AUTHOR = {Deligne, P. and Mostow, G.},
     TITLE = {Monodromy of hypergeometric functions and nonlattice integral
              monodromy},
   JOURNAL = {Inst. Hautes \'Etudes Sci. Publ. Math.},
  FJOURNAL = {Institut des Hautes \'Etudes Scientifiques. Publications
              Math\'ematiques},
    NUMBER = {63},
      YEAR = {1986},
     PAGES = {5--89},
      ISSN = {0073-8301,1618-1913},
   MRCLASS = {22E40 (32G20 33A30)},
MRREVIEWER = {Shigeaki\ Tsuyumine},
       URL = {http://www.numdam.org/item?id=PMIHES_1986__63__5_0},
}

@Article{FLNU_StringyChernClasses,
 Author = {{De Fernex}, Tommaso and Lupercio, Ernesto and Nevins, Thomas and Uribe, Bernardo},
 Title = {Stringy {Chern} classes of singular varieties},
 FJournal = {Advances in Mathematics},
 Journal = {Adv. Math.},
 ISSN = {0001-8708},
 Volume = {208},
 Number = {2},
 Pages = {597--621},
 Year = {2007},
 DOI = {10.1016/j.aim.2006.03.005},
 zbMATH = {5083643},
 Zbl = {1113.14008}
}

@book {Fulton_ToricVarieties,
    AUTHOR = {Fulton, William},
     TITLE = {Introduction to toric varieties},
    SERIES = {Annals of Mathematics Studies},
    VOLUME = {131},
      NOTE = {The William H. Roever Lectures in Geometry},
 PUBLISHER = {Princeton University Press, Princeton, NJ},
      YEAR = {1993},
     PAGES = {xii+157},
      ISBN = {0-691-00049-2},
   MRCLASS = {14M25 (14-02 14J30)},
MRREVIEWER = {T.\ Oda},
       DOI = {10.1515/9781400882526},
       URL = {https://doi.org/10.1515/9781400882526},
}

@article {GKPT_MY_Inequality_Uniformisation_of_Canonical_models,
    AUTHOR = {Greb, Daniel and Kebekus, Stefan and Peternell, Thomas and
              Taji, Behrouz},
     TITLE = {The {M}iyaoka-{Y}au inequality and uniformisation of canonical
              models},
   JOURNAL = {Ann. Sci. \'Ec. Norm. Sup\'er. (4)},
  FJOURNAL = {Annales Scientifiques de l'\'Ecole Normale Sup\'erieure.
              Quatri\`eme S\'erie},
    VOLUME = {52},
      YEAR = {2019},
    NUMBER = {6},
     PAGES = {1487--1535},
      ISSN = {0012-9593,1873-2151},
}

@incollection{GKT_OverviewUniformisation,
 author = {Greb, Daniel and Kebekus, Stefan and Taji, Behrouz},
 title = {Uniformisation of higher-dimensional minimal varieties},
 booktitle = {{P}roceedings of the 2015 summer research institute in algebraic geometry, {U}niversity of {U}tah, {S}alt {L}ake {C}ity, {UT}, {USA}, {P}art 1. {A}merican {M}athematical {S}ociety ({AMS}), {P}rovidence, {RI} and {C}lay {M}athematics {I}nstitute, {C}ambridge, {M}A},
 isbn = {978-1-4704-3577-6; 978-1-4704-4667-3; 978-1-4704-4678-9},
 pages = {277--308},
 year = {2018},
 publisher = {Providence, RI: American Mathematical Society (AMS); Cambridge, MA: Clay Mathematics Institute},
}

@Article{GKPT_HarmonicMetricsUniformisation,
 Author = {Greb, Daniel and Kebekus, Stefan and Peternell, Thomas and Taji, Behrouz},
 Title = {Harmonic metrics on {Higgs} sheaves and uniformization of varieties of general type},
 FJournal = {Mathematische Annalen},
 Journal = {Math. Ann.},
 ISSN = {0025-5831},
 Volume = {378},
 Number = {3-4},
 Pages = {1061--1094},
 Year = {2020},
 DOI = {10.1007/s00208-019-01906-4},
}

@book {Har77,
    AUTHOR = {Hartshorne, Robin},
     TITLE = {Algebraic geometry},
    SERIES = {Graduate Texts in Mathematics},
    VOLUME = {No. 52},
 PUBLISHER = {Springer-Verlag, New York-Heidelberg},
      YEAR = {1977},
     PAGES = {xvi+496},
}

@Article{Hunt_GeographyOfComplexSurfacesAnd3Folds,
 Author = {Hunt, Bruce},
 Title = {Complex manifold geography in dimension 2 and 3},
 FJournal = {Journal of Differential Geometry},
 Journal = {J. Differ. Geom.},
 ISSN = {0022-040X},
 Volume = {30},
 Number = {1},
 Pages = {51--153},
 Year = {1989},
 DOI = {10.4310/jdg/1214443287},
 zbMATH = {4171097},
 Zbl = {0712.14022}
}

@article{KO80,
 author = {Kobayashi, Shoshichi and Ochiai, Takushiro},
 title = {Holomorphic projective structures on compact complex surfaces},
 fjournal = {Mathematische Annalen},
 journal = {Math. Ann.},
 issn = {0025-5831},
 volume = {249},
 pages = {75--94},
 year = {1980},
}

@article{KollarMatsusaka_RRInequalities,
 author = {Koll{\'a}r, J. and Matsusaka, T.},
 title = {Riemann-{Roch} type inequalities},
 fjournal = {American Journal of Mathematics},
 journal = {Am. J. Math.},
 issn = {0002-9327},
 volume = {105},
 pages = {229--252},
 year = {1983},
 doi = {10.2307/2374387},
 zbMATH = {3855264},
 Zbl = {0538.14006}
}

@book {kollarMori_BirationalGeometry,
    AUTHOR = {Koll\'{a}r, J\'{a}nos and Mori, Shigefumi},
     TITLE = {Birational geometry of algebraic varieties},
    SERIES = {Cambridge Tracts in Mathematics},
    VOLUME = {134},
 PUBLISHER = {Cambridge University Press, Cambridge},
      YEAR = {1998}
}

@book {lazarsfeld_PositivityI,
    AUTHOR = {Lazarsfeld, Robert},
     TITLE = {Positivity in algebraic geometry. {I}},
    SERIES = {Ergebnisse der Mathematik und ihrer Grenzgebiete. 3. Folge. A
              Series of Modern Surveys in Mathematics},
    VOLUME = {48},
      NOTE = {Classical setting: line bundles and linear series},
 PUBLISHER = {Springer-Verlag, Berlin},
      YEAR = {2004}
}

@article{LU_RRTypeInequalities,
 author = {Luo, Tie},
 title = {Riemann-{Roch} type inequalities for nef and big divisors},
 fjournal = {American Journal of Mathematics},
 journal = {Am. J. Math.},
 issn = {0002-9327},
 volume = {111},
 number = {3},
 pages = {457--487},
 year = {1989},
 doi = {10.2307/2374669},
 zbMATH = {4130532},
 Zbl = {0691.14005}
}

@article{LTZ_CanonicalClassInequalityFibredThreefolds,
 author = {Lu, Jun and Tan, Sheng-Li and Zuo, Kang},
 title = {Canonical class inequality for fibred spaces},
 fjournal = {Mathematische Annalen},
 journal = {Math. Ann.},
 issn = {0025-5831},
 volume = {368},
 number = {3-4},
 pages = {1311--1332},
 year = {2017},
 doi = {10.1007/s00208-016-1474-2},
 zbMATH = {6774921},
 Zbl = {1401.14054}
}

@article{LuXiao_InequalitiesOfRiemannRochType,
 author = {Lu, Xing and Xiao, Jian},
 title = {The inequalities of {Chern} classes and {Riemann}-{Roch} type inequalities},
 fjournal = {Advances in Mathematics},
 journal = {Adv. Math.},
 issn = {0001-8708},
 volume = {458},
 pages = {24},
 note = {Id/No 109982},
 year = {2024},
 doi = {10.1016/j.aim.2024.109982},
}

@article{Luna73,
 author = {Luna, Domingo},
 title = {Slices {\'e}tal{\'e}s},
 fjournal = {Bulletin de la Soci{\'e}t{\'e} Math{\'e}matique de France. Suppl{\'e}ment. M{\'e}moires},
 journal = {Bull. Soc. Math. Fr., Suppl., M{\'e}m.},
 issn = {0583-8665},
 volume = {33},
 pages = {81--105},
 year = {1973},
}

@book {Milnor_TopologyDifferentialTopology,
    AUTHOR = {Milnor, John},
     TITLE = {Topology from the differentiable viewpoint},
      NOTE = {Based on notes by David W. Weaver},
 PUBLISHER = {University Press of Virginia, Charlottesville, VA},
      YEAR = {1965},
     PAGES = {ix+65},
   MRCLASS = {57.10},
}

@article{Miyaoka_ChernNumbersSurfacesGeneralType,
 author = {Miyaoka, Yoichi},
 title = {On the {Chern} numbers of surfaces of general type},
 fjournal = {Inventiones Mathematicae},
 journal = {Invent. Math.},
 issn = {0020-9910},
 volume = {42},
 pages = {225--237},
 year = {1977},
 doi = {10.1007/BF01389789},
 url = {https://eudml.org/doc/142501},
 zbMATH = {3582298},
 Zbl = {0374.14007}
}

@misc{Mul_Thesis,
 author = {Niklas M{\"u}ller},
 title = {Inequalities of {Miyaoka}-{Yau} type {$\&$} {Uniformisation} of varieties of intermediate {Kodaira} {Dimension}},
 year = {2026},
 note = {Preprint \href{https://doi.org/10.48550/arXiv.2601.15138}{arXiv:2601.15138}},
}

@Article{PrasadYeung_FakeProjectivePlanes,
 Author = {Prasad, Gopal and Yeung, Sai-Kee},
 Title = {Fake projective planes},
 FJournal = {Inventiones Mathematicae},
 Journal = {Invent. Math.},
 ISSN = {0020-9910},
 Volume = {168},
 Number = {2},
 Pages = {321--370},
 Year = {2007},
 DOI = {10.1007/s00222-007-0034-5},
 zbMATH = {5151314},
 Zbl = {1253.14034}
}

@article{PF_ChernCLassesPolarisedManifolds,
 author = {Li, Ping and Zheng, Fangyang},
 title = {Chern class inequalities on polarized manifolds and nef vector bundles},
 fjournal = {IMRN. International Mathematics Research Notices},
 journal = {Int. Math. Res. Not.},
 issn = {1073-7928},
 volume = {2022},
 number = {8},
 pages = {6262--6288},
 year = {2022},
 doi = {10.1093/imrn/rnaa317},
 zbMATH = {7510521},
 Zbl = {1495.32055}
}

@article{SchreiederTasin_ChernClassesUnbounded,
 author = {Schreieder, Stefan and Tasin, Luca},
 title = {Algebraic structures with unbounded {Chern} numbers},
 fjournal = {Journal of Topology},
 journal = {J. Topol.},
 issn = {1753-8416},
 volume = {9},
 number = {3},
 pages = {849--860},
 year = {2016},
 doi = {10.1112/jtopol/jtw011},
 zbMATH = {6640415},
 Zbl = {1405.32035}
}

@article {Yau_KEMetrics,
    AUTHOR = {Yau, Shing Tung},
     TITLE = {Calabi's conjecture and some new results in algebraic
              geometry},
   JOURNAL = {Proc. Nat. Acad. Sci. U.S.A.},
  FJOURNAL = {Proceedings of the National Academy of Sciences of the United
              States of America},
    VOLUME = {74},
      YEAR = {1977},
    NUMBER = {5},
     PAGES = {1798--1799},
      ISSN = {0027-8424},
}

@article{ZaslowStringyInvariants,
    AUTHOR = {Zaslow, E.},
     TITLE = {Topological orbifold models and quantum cohomology rings},
   JOURNAL = {Commun. Math. Phys.},
    VOLUME = {156},
      YEAR = {1993},
     PAGES = {301–-331},
}

@Misc{DM_ToledoInvariantsTQFT,
 Author = {Deroin, Bertrand and March{\'e}, Julien},
 Title = {Toledo invariants of {Topological} {Quantum} {Field} {Theories}},
 Year = {2022},
note = {Preprint \href{https://doi.org/10.48550/arXiv.2207.09952}{arXiv:2207.09952}},
}

@misc{DuSun_ChernClassInequalities,
 author = {Du, Rong and Sun, Hao},
 title = {Inequalities of {Chern} classes on nonsingular projective $n$-folds of {Fano} and general type with ample canonical bundle},
 year = {2017},
 note = {Preprint \href{https://doi.org/10.48550/arXiv.1712.03458}{arXiv:1712.03458}},
}

@misc{JR02,
 author = {Jahnke, Priska and Radloff, Ivo},
 title = {On manifolds with holomorphic normal projective connections},
 year = {2002},
 note = {Preprint \href{https://doi.org/10.48550/arXiv.0903.4571}{arXiv:0903.4571}},
}

@misc{Sun_bogomolovtypeinequalitiesfrobenius,
      title={Bogomolov type inequalities and {F}robenius semipositivity}, 
      author={Hao Sun},
      year={2025},
      note = {Preprint \href{https://doi.org/10.48550/arXiv.2504.16829}{arXiv:2504.16829}}
}
